\documentclass[reqno, a4paper]{amsart}

\usepackage{graphicx} 
\usepackage[hidelinks]{hyperref}
\hypersetup{colorlinks=true,linkcolor=blue,citecolor=blue}
\usepackage[utf8]{inputenc}
\usepackage[german, english]{babel}
\usepackage[T2A,OT2,OT1]{fontenc}

\newcommand{\lc}{\preceq_{\textnormal{c}}}

\newcommand{\ld}{\preceq_{E}}

\newcommand{\lic}{\preceq_{icx}}

\usepackage{datetime}

\newdateformat{monthyeardate}{%
  \monthname[\THEMONTH], \THEYEAR}
\usepackage{amsfonts}

\usepackage{bbm}
\usepackage{amssymb}
\usepackage{amsthm}
\usepackage[dvipsnames]{xcolor}
\newtheorem{theorem}{Theorem}[section]

\newtheorem{lemma}[theorem]{Lemma}
\usepackage{mathtools}
\usepackage{accents}
\usepackage{enumitem}

\newcommand{\conv}{\mathrm{conv}\,}
\newcommand{\convr}{\mathrm{conv}_R\,}

\newcommand{\iconv}{\mathrm{iconv}\,}
\newcommand{\PP}{\mathcal{P}}

\newcommand{\cP}{\mathcal{P}}
\newcommand{\Pc}{\cP}
\renewcommand{\P}{\cP}

\newcommand{\midauto}{\,\middle\vert\,}
\newcommand{\R}{\mathbb R}
\theoremstyle{definition}
\newtheorem{definition}[theorem]{Definition}
  \theoremstyle{definition}

\theoremstyle{definition}
\newtheorem{remark}[theorem]{Remark}

\newcommand{\cpl}{\Pi}

\newcommand{\LSC}{\mathcal{LSC}}
\newcommand{\mean}{\mathrm{mean}}

\title{Weak Optimal Transport: When is the Dual Potential convex?}
\author{Filip Pramenković}
\date{\monthyeardate\today}

\begin{document}

\begin{abstract}
 Weak optimal transport generalizes the classical theory of optimal transportation to nonlinear cost functions and covers a range of problems that lie beyond the traditional theory --- including entropic transport, martingale transport, and applications in mechanism design. As in the classical case, the weak transport problem can also be written as a dual maximization problem over a pair of conjugate potentials. 

 We identify sharp monotonicity conditions on the cost under which the dual problem can be restricted to \emph{convex} potentials. This framework unifies several known results from the literature, including barycentric transport, martingale Benamou--Brenier, the multiple-good monopolist problem, Strassen's theorem, stochastic order projections and the classical Brenier theorem. 

\end{abstract}

\maketitle

\section{Introduction}

Weak optimal transport was introduced by Gozlan, Roberto, Samson and Tetali \cite{GoRoSaTe14} to extend classical transport theory to nonlinear cost functions. It  encompasses important applications that fall outside the classical transport framework while still allowing for the same basic existence and duality results, see e.g.\  \cite{GoRoSaSh18, AlBoCh18, BaPa19, BePaRiSc25}.

To introduce the weak transport problem, we set up some notation.
We consider Polish metric spaces $(X,d_X), (Y, d_Y)$ equipped with Borel probability measures $\mu, \nu$, respectively.
We write $\cpl(\mu, \nu)$ for the set of all \emph{couplings} or \emph{transport plans}, that is, probabilities on \(X \times Y\) which have \(X\)-marginal \(\mu\) and \(Y\)-marginal \(\nu\). For $p\geq 1$ we write $\Pc_p(Y)$ for the set of probabilities with finite $p$-moment, equipped with $p$-weak convergence. 
Given a lower semicontinuous (lsc) cost function $C:X\times \Pc_p(Y)\to [0,\infty]$ that is convex in the second argument, the primal weak transport problem is given by
\begin{align}\label{eq:PrimalWOT}
    \inf_{\pi\in\Pi(\mu,\nu)}\int C(x,\pi_x)d\mu(x),
\end{align}
where $\pi_x$ denotes the disintegration of $\pi$ with respect to its first marginal. As in classical optimal transport, lower semicontinuity of $C$ implies that \eqref{eq:PrimalWOT} is attained and equals the dual problem
\begin{align}
    \sup_{\psi \in {\mathcal{C}}_{b,p}(Y)} \mu(\psi^C) - \nu(\psi), \quad \psi^C(x):= \inf_{\rho\in \Pc_p(Y)}\rho(\psi) + C(x,\rho). 
\end{align}
We write $\mathcal C_{b,p}(Y)$ for the subset of continuous functions bounded from below by a constant and bounded from above by a multiple of $1+d_Y(\cdot,y_0)^p$ for some $y_0\in Y$.  

In many applications of WOT-duality, the dual problem simplifies significantly. For example, \ given $\mu, \nu\in \Pc_1(\R^d)$,  we have 
\begin{align}
    \inf_{\pi\in \cpl(\mu, \nu)} \int \chi_{x= \mean(\pi_x)} \, d\mu(x) = \sup_{ \substack{\psi \textrm{ convex} \\ \psi \in \mathcal C_{b,1}(\R^d)} } \mu(\psi)- \nu(\psi),
\end{align}
where $\chi_{x= \mean(\pi_x)}=0$, if $x= \mean(\pi_x)$, and $\infty$ otherwise. This duality result was first noticed in \cite{GoRoSaTe14} and directly implies Strassen's theorem on the existence of martingales between marginals in convex order, see Theorem \ref{strass} below for details.

To further illustrate the point, for $\mu, \nu \in \Pc_1(\R)$ we have 
\begin{align}
    \inf_{\pi\in \cpl(\mu, \nu)} \int (x-\mean(\pi_x))_+ \, d\mu(x) = \sup_{\substack{\psi \textrm{ increasing convex}\\ 1\textrm{-Lipschitz}}} \mu(\psi)- \nu(\psi).
\end{align}
This is referred to as the increasing convex Kantorovich--Rubinstein formula in \cite{BePaRiSc25}. 

We highlight that in both examples it suffices to optimize  over a smaller class of functions in the dual problem, namely convex and increasing convex functions. In fact, the same is true for several further problems in the weak transport framework: the original convex Kantorovich--Rubinstein formula \cite{GoRoSaTe14}, the Brenier--Strassen theorem \cite{GoJu18}, the dual formulation of the martingale Benamou--Brenier problem \cite{backhoffveraguas2025existencebassmartingalesmartingale}, optimal mechanism design \cite{DaDeTz17}, and the Gangbo--McCann--Strassen theorem \cite{BePaRiSc25}. 

The purpose of this note is to provide a unified explanation of the aforementioned dual problem simplifications. We specialize to the case of $Y=\R^d$, which coincides with the listed examples and simplifies our discussion. It turns out that \textit{stochastic orders} are the correct tool for explaining this phenomenon. Recall that $\rho_1, \rho_2 \in\PP_p(\R^d)$ are in  \textit{convex order}, in signs $\rho_1\lc \rho_2$, if $\rho_1(f) \leq \rho_2(f)$ for all convex functions $f:\R^d\to\R$. 

Our main result relates convexity of the dual potentials in weak transport problems to a monotonicity condition for the cost function:
\begin{theorem} \label{mainthm}
    Let $\mu\in\PP(X)$, $\nu \in \PP_p(\R^d)$ and $C:X\times\PP_p(\R^d)\to[0,\infty]$ be lsc and convex in the second argument. If $C$ is $\lc$-decreasing in the second argument, then 
\begin{align} \label{mainthmdispl}
		\inf_{\pi\in\Pi(\mu,\nu)} \int C(x,\pi_x)d\mu(x)= \sup_{\substack{\psi \in  {\mathcal{C}}^{}_{b,p}(\R^d) \\ \psi\textnormal{ convex}}} \mu(\psi^C) - \nu(\psi).
	\end{align}
    Conversely, if (\ref{mainthmdispl}) holds for all $\mu\in\PP(X)$, $\nu\in\PP_p(\R^d)$, then $C$ is $\lc$-decreasing in the second argument.
\end{theorem}

Furthermore, in Theorem \ref{thmatt} we show that under suitable regularity assumptions, the supremum in (\ref{mainthmdispl}) is attained by a $\nu$-integrable convex function $\psi^{\text{opt}}:\R^d\to(-\infty,\infty]$. More precisely, to guarantee attainment we require that the cost function does not grow too rapidly and that it exhibits some mild continuity in the second argument, see Section \ref{sec att}.

Analogous duality and attainment results hold for coordinatewise increasing convex functions and the associated \textit{increasing convex order} denoted by $\lic$, see Theorem \ref{mainthmic} and Theorem \ref{thmatt2}. 

In fact, the simplification of the dual problem can be extended to general Polish spaces and orders defined by \textit{stable cones}, an abstract class of functions that captures the relevant properties of convex and increasing convex functions, see Section \ref{nesto}.

\medskip

The paper is organized as follows. Section \ref{PR} serves as a preparatory section, where we collect relevant preliminary results for our discussion. In Section \ref{MT}, we prove the central theorems on restricting the dual to (increasing) convex functions and in Section \ref{sec att} we consider under what assumptions and in what sense the dual problem is attained.  In Section \ref{APP} we apply the main theorems to recover known results from the literature and in Section \ref{nesto} we discuss a general abstract framework for the restriction problem.



\section{Preliminary Results}
\label{PR}
In this section, we lay the foundations of our main theorems by discussing and proving some auxiliary results. 


\subsection{Convex and Increasing Convex Hulls}
 We start by recalling the definitions of the convex hull and the increasing convex hull, which play a key role in proving the main theorems.

    Let $\psi\in{\mathcal{C}}_{b,p}(\R^d)$. We denote its \textit{convex hull} as 
    $$\conv\psi(y):=\sup_{ \substack{\phi\le\psi \\ \phi \text{ convex}}} \phi(y), \;\; y\in\R^d,$$
and its \textit{increasing convex hull} as 
    $$\iconv\psi(y):=\sup_{ \substack{\phi\le\psi \\ \phi \text{ inc. convex}}} \phi(y), \;\; y\in\R^d.$$


    

Recall that for $\rho_1,\rho_2\in\P_p(\R^d)$ the notation $\rho_1\lc \rho_2$ (resp. $\rho_1\lic \rho_2$), means that $\rho_1(f)\le\rho_2(f)$ for all $f:\R^d\to\R$ convex (resp. increasing convex).

The following lemma is a collection of well-known facts about the convex and increasing convex hull.

\begin{lemma} \label{repr lemm}
    Let $\psi \in {\mathcal{C}}_{b,p}(\R^d)$. Then we have 
    \begin{enumerate}[label=(\roman*)]\label{repr enac}
     \item \label{conv repr bullet}  $\conv{\psi}(y)=\inf\left\{\rho(\psi) \mid \delta_y\lc\rho,\; \rho\in\PP_p(\R^d)\right\},$
     \item \label{iconv repr bullet}  $\iconv{\psi}(y)=\inf\left\{\rho(\psi) \mid \delta_y\lic\rho,\; \rho\in\PP_p(\R^d)\right\},$
      \item \label{prop1} $\conv{\psi}\leq \psi$,
     \item \label{prop3}$\psi_1,\psi_2\in{\mathcal{C}}_{b,p}(\R^d)$ and $\psi_1\leq\psi_2$, then $\conv{\psi_1}\leq \conv{\psi_2},$
    \item \label{prop4}$\conv{\psi}\in{\mathcal{C}}_{b,p}(\R^d),$
    \item \label{prop6}$\conv{\psi_n}\downarrow \conv{\psi}$, as $n\to\infty$, where $\psi_n:=\psi+\frac{1}{n}|\cdot|^p.$
    \end{enumerate}
Additionally, properties \ref{prop1}--\ref{prop6} also hold for the increasing convex hull.
\end{lemma}
\begin{proof}
Since \textit{\ref{iconv repr bullet}} is less commonly cited than \textit{\ref{conv repr bullet}}, let us prove it. 
Fix $\psi\in{\mathcal{C}}_{b,p}(\R^d)$, $y\in \R^d$ and $\rho\in\P_p(\R^d)$ such that $\delta_y \lic \rho$. Then
$$\iconv\psi(y)=\sup_{ \substack{\phi\le\psi \\ \phi \text{ inc. convex}}} \phi(y)\leq \sup_{ \substack{\phi\le\psi \\ \phi \text{ inc. convex}}} \rho(\phi) \leq \rho(\psi),$$
and taking the infimum over all such $\rho$ proves the first direction. 

Now write $$f(y):=\inf\left\{\rho(\psi) \mid \delta_y\lic\rho,\; \rho\in\PP_p(\R^d)\right\}.$$
To prove the converse inequality, it suffices to show that $f$ is increasing, convex and majorized by $\psi$.

Note that choosing $\rho=\delta_y$ yields $f(y)\leq\psi(y)$. Next, we show that $f$ is convex. Fix $y_1,y_2\in\R^d$, $\alpha\in(0,1)$ and $\varepsilon>0$. Pick $\rho_1,\rho_2 \in\P_p(\R^d)$ with  $\delta_{y_1}\lic \rho_1$ and  $\delta_{y_2}\lic \rho_2$ such that $\rho_1(\psi)\leq f(y_1)+\varepsilon$ and $\rho_2(\psi)\leq f(y_2)+\varepsilon$. Then
$$f(\alpha y_1+(1-\alpha)y_2)\leq (\alpha\rho_1+(1-\alpha)\rho_2)(\psi) \leq \alpha f(y_1)+(1-\alpha)f(y_2)+ \varepsilon.$$

Lastly, we show that $f$ is increasing. Let $y_1,y_2\in\R^d$ be such that $y_1\leq y_2$ coordinatewise. In particular, we have $\delta_{y_1}\lic\delta_{y_2}$. But by definition of $f$, this immediately implies $f(y_1)\leq f(y_2).$ 

Points \textit{\ref{prop1}--\ref{prop6}} are obvious.
\end{proof}

\subsection{Weak Optimal Transport and Duality}\label{WOTD}
We now introduce some key notions related to weak optimal transport. As is common in weak transport, we assume that $C:X\times\PP_p(Y)\to[0,\infty]$ is lsc and the map $\rho\mapsto C(x,\rho)$ is convex for each $x\in X$.


    


\begin{definition}[$C$-conjugate] \label{cconj}
   Let $C:X\times\PP_p(Y)\to[0,\infty]$ be measurable. We define the $C$-conjugate $\psi^C:X\to (-\infty,\infty]$ by
    \begin{equation*}
        \psi^C(x):=\inf_{\rho\in\PP_p(Y)} \rho(\psi)+C(x,\rho), \;\;\psi\in{\mathcal{C}}^{}_{b,p}(Y).
    \end{equation*}
\end{definition}
The function $\psi^C$ is lower semianalytic, in particular, universally measurable (see \cite[Proposition 7.47]{bertsekas1978stochastic}). It is also bounded from below, hence $\mu(\psi^C)$ is well-defined for all measures $\mu\in\PP(X)$.

\begin{lemma} \label{rc stabil lemma}
    Let $C:X\times\PP_p(\R^d)\to[0,\infty]$ be measurable and $\psi\in {\mathcal{C}}_{b,p}(\R^d)$. If $C$ is $\lc$-decreasing in the second argument, then we have $$ \psi^C=({\conv{\psi}})^C,$$
    and if $C$ is $\lic$-decreasing in the second argument, then we have $$ \psi^C=({\iconv{\psi}})^C.$$
\end{lemma}

\begin{proof} We prove the first claim, as the second is proven identically.
    Fix $\psi\in{\mathcal{C}}_{b,p}(\R^d)$. Then
    $$ \psi^C(x)=\inf_{\rho\in\PP_p(\R^d)} \rho(\psi)+C(x,\rho)\geq \inf_{\rho\in\PP_p(\R^d)} \rho(\conv\psi)+C(x,\rho)=(\conv{\psi})^C(x),$$
    where the inequality is by Lemma \ref{repr lemm}\textit{\ref{prop1}}. 
    
    The converse inequality is equivalent to the statement
   \begin{equation} \label{nejjed}
        \forall\varepsilon>0\; \forall\rho\in\PP_p(\R^d) \;\exists\widetilde{\rho}\in\PP_p(\R^d):\widetilde{\rho}(\psi)+C(x,\widetilde{\rho})\leq {\rho}(\conv{\psi})+C(x,\rho)+\varepsilon.
   \end{equation}
    Fix $\varepsilon>0$ and $\rho\in\PP_p(\R^d).$ Then by Lemma \ref{repr lemm}\textit{\ref{prop6}} we have
    $\conv{\psi_n}\downarrow \conv{\psi}$, as $n\to\infty$, where $\psi_n:=\psi+\frac{1}{n}|\cdot|^p$. Since $\rho(\conv{\psi_1})<\infty$, we have by monotone convergence 
    $$\rho(\conv{\psi_n})\xrightarrow{n\to\infty}\rho(\conv{\psi}).$$
    So fix $N\in\mathbb N$ sufficiently large such that
   \begin{equation} \label{suff small}
        \rho(\conv{\psi_N})-\rho(\conv \psi)\leq \varepsilon/2.
   \end{equation}
    Now by Lemma \ref{repr lemm}\textit{\ref{conv repr bullet}} applied to $\psi_N$ we have that the set 
    $$A:=\left\{(y,\xi)\in \R^d\times\PP_p(\R^d) \mid \conv{\psi_N}(y)+\varepsilon/2\geq \xi(\psi_N), \;\delta_y \lc\xi\right\}$$
    satisfies $\text{proj}_1A=\R^d$. It is straightforward to verify that $A$ is Borel measurable and by the von Neumann uniformization theorem, we may extract a universally measurable selection $\{\xi_y\}_{y\in \R^d}\subseteq\PP_p(\R^d)$, such that
    \begin{equation} \label{selec prop}
        \conv{\psi_N}(y)+\varepsilon/2\geq \xi_y(\psi_N), \;\delta_y \lc\xi_y.
    \end{equation}
   We define $\widetilde{\rho}$ by $d\widetilde{\rho}(z):=\int_{\R^d} d\xi_y(z)d\rho(y).$ Integrating the inequality in (\ref{selec prop}) with respect to $\rho$ yields
   \begin{equation} \label{bitna}
       \rho(\conv{\psi_N})+\varepsilon/2\geq \widetilde{\rho}(\psi_N).
   \end{equation}
   On the one hand, since $\psi$ is bounded from below we have
$$\widetilde{\rho}(\psi_N)\geq \ell +1/N\widetilde{\rho}(|\cdot|^p),$$
and on the other hand, by Lemma \ref{repr lemm}\textit{\ref{prop4}}, 
$\rho(\conv{\psi_N}) <\infty.$
These two inequalities combined with (\ref{bitna}) yield that $\widetilde{\rho}\in\PP_p(\R^d)$.

Now notice that (\ref{suff small}) and (\ref{bitna}) imply
\begin{equation} \label{bistna}
       \rho(\conv{\psi})+\varepsilon\geq \widetilde{\rho}(\psi_N)\geq\widetilde{\rho}(\psi).
   \end{equation}
This partially shows the inequality in (\ref{nejjed}). To finalize the proof, note that for arbitrary convex $f$, we have
$$\widetilde{\rho}(f)=\int \xi_y(f)d\rho(y)\geq\int \delta_y(f)d\rho(y)=\rho(f),$$
i.e., $\rho\lc\widetilde{\rho}$, which by the monotonicity assumption implies $C(x,\rho)\geq C(x,\widetilde{\rho})$.
\end{proof}

We conclude this section by recalling the general duality theorem for weak optimal transport.
\begin{theorem}[{\cite[{Theorem 1.3}]{veraguas2019existence}}]\label{th:kantorovichduality}
	 Let $\mu\in\PP(X)$, $\nu \in \PP_p(Y)$ and $C:X\times\PP_p(Y)\to[0,\infty]$ be lsc and convex in the second argument. Then
	\begin{align}
		\inf_{\pi\in\Pi(\mu,\nu)} \int C(x,\pi_x)d\mu(x) = \sup_{\psi \in {\mathcal{C}}_{b,p}(Y)} \mu(\psi^C) - \nu(\psi).
	\end{align}
\end{theorem}

\section{Proof of the Main Theorems}\label{MT}

We now prove the main theorem on restricting the dual problem to convex functions.

\begin{proof}[Proof of Theorem \ref{mainthm}]
    Assume that $C$ is $\lc$-decreasing. Fix $\mu\in\PP(X)$ and $\nu\in\PP_p(\R^d)$. Then we have
    \begin{align*}
		\inf_{\pi\in\Pi(\mu,\nu)} \int C(x,\pi_x)d\mu(x)&= \sup_{\psi \in {\mathcal{C}}^{}_{b,p}(\R^d)} \mu(\psi^C) - \nu(\psi)\\
        &\leq \sup_{\psi \in {\mathcal{C}}^{}_{b,p}(\R^d)} \mu( (\conv\psi)^C) - \nu(\conv\psi)
	\end{align*}
    where the equality holds by Theorem \ref{th:kantorovichduality} and the inequality holds by Lemma \ref{rc stabil lemma} and Lemma \ref{repr lemm}\textit{\ref{prop1}}. 
But whenever $\psi\in{\mathcal{C}}_{b,p}(\R^d)$, we have $\conv \psi \in {\mathcal{C}}_{b,p}(\R^d)$, so
    \begin{align*}
        &\leq \sup_{\substack{\psi \in {\mathcal{C}}^{}_{b,p}(\R^d) \\ \psi \text{ convex}}} \mu( \psi^C) - \nu(\psi)\\
        &\leq \sup_{\psi \in {\mathcal{C}}^{}_{b,p}(\R^d)} \mu(\psi^C) - \nu(\psi)\\&=\inf_{\pi\in\Pi(\mu,\nu)} \int C(x,\pi_x)d\mu(x).
	\end{align*}

    For the converse, fix $x\in X$ and $\rho_1,\rho_2\in\PP_p(\R^d)$, such that $\rho_1\lc\rho_2$. By assumption, we have
    \begin{align*} 
       C(x,\rho_1)&=\inf_{\pi\in\Pi(\delta_x,\rho_1)}\int C(\widetilde x,\pi_{\widetilde{x}}) d\delta_x(\widetilde{x})= \sup_{\substack{\psi \in  {\mathcal{C}}^{}_{b,p}(\R^d) \\\psi \text{ convex}}} \psi^C(x) - \rho_1(\psi)\\
       &\geq \sup_{\substack{\psi \in  {\mathcal{C}}^{}_{b,p}(\R^d) \\ \psi\text{ convex}}} \psi^C(x) - \rho_2(\psi)=C(x,\rho_2).
    \end{align*}
\end{proof}
The increasing convex version of the restriction theorem also holds and is proven analogously.
\begin{theorem} \label{mainthmic}
    Let $\mu\in\PP(X)$, $\nu \in \PP_p(\R^d)$ and $C:X\times\PP_p(\R^d)\to[0,\infty]$ be lsc and convex in the second argument. If $C$ is $\lic$-decreasing in the second argument, then 
\begin{align} \label{mainthmdisplic}
		\inf_{\pi\in\Pi(\mu,\nu)} \int C(x,\pi_x)d\mu(x)= \sup_{\substack{\psi \in  {\mathcal{C}}^{}_{b,p}(\R^d) \\ \psi\textnormal{ inc. convex}}} \mu(\psi^C) - \nu(\psi).
	\end{align}
    Conversely, if (\ref{mainthmdisplic}) holds for all $\mu\in\PP(X)$, $\nu\in\PP_p(\R^d)$, then $C$ is $\lic$-decreasing in the second argument.
\end{theorem}

    The proof of Theorem \ref{mainthm} (and the auxiliary results it uses) can be modified to accommodate cost functions $C$ that take negative values, for example, the martingale Benamou--Brenier problem \cite{backhoffveraguas2025existencebassmartingalesmartingale}. This relaxation is formulated in the following theorem.

\begin{theorem} \label{thmmod1}
   Let $\mu\in\PP(X)$, $\nu \in \PP_p(\R^d)$, and $C:X\times\PP_p(\R^d)\to(-\infty,\infty]$ be lsc, convex in the second argument, and such that it satisfies the lower bound $C(x,\rho)\geq-(a_\ell(x)+\rho(b_\ell))$, for some $a_\ell\in L^1(\mu)\cap \LSC(X)$\footnote{$\LSC(X)$ denotes the real-valued lsc functions on $X$} and some convex function $b_\ell\in{\mathcal{C}}_{b,p}(\R^d)$. If $C$ is $\lc$-decreasing in the second argument, then 
    \begin{align*} 
		\inf_{\pi\in\Pi(\mu,\nu)} \int C(x,\pi_x)d\mu(x)= \sup_{\substack{{\psi \in ({\mathcal{C}}_{b,p}(\R^d)+b_\ell)}\\ \psi \textnormal{ convex}}} \mu(\psi^C) - \nu(\psi).
	\end{align*}
\end{theorem}
The analogous result also holds for the increasing convex case.

\section{Attainment considerations}
\label{sec att}
While duality holds in general, the supremum is only attained under additional regularity assumptions.  Dealing with the delicate question of dual attainment requires us to define the dual problem appropriately, so we follow the treatment in \cite{BePaRiSc25}. Let $X,Y$ be Polish spaces, $\mu\in\P(X)$ and $\nu\in\PP_p(Y)$. Let us denote by $\mathcal L^1(\rho)$ the space of $(-\infty,\infty]$-valued Borel functions that are $\rho$-integrable.

To guarantee attainment, we shall require two conditions. Firstly, in line with the classical theory, we require a \textit{boundedness condition}
\begin{equation*}
    \label{condB} \tag{B}
    C(x,\rho)\leq a(x)+\rho(b)
\end{equation*}
for a $\mu$-integrable function $a:X\to\R$ and a $\nu$-integrable function $b:Y\to\R$.

Secondly, we require a mild \textit{continuity condition}, that is, for any increasing sequence $(Y_k)_k$ of Borel sets with $\bigcup_kY_k=Y$, we have 
\begin{equation*}
    \label{condC} \tag{C}
    C(x,\rho)\geq \limsup_k C(x,(\rho|Y_k)/\rho(Y_k)).
\end{equation*}

Let us emphasize the subtle difference between Definition \ref{cconj} and the way we define the $C$-conjugate in this section. For a measurable function $\psi:Y\to(-\infty,\infty]$, we have
$$\psi^C(x):=\inf_{\substack{\rho\in\PP_p(Y)\\ \text{s.t. }\psi\in\mathcal L ^1(\rho)} } \rho(\psi)+C(x,\rho),$$
where the infimum over an empty set is interpreted as $+\infty$. In the case of $\psi\in{\mathcal{C}}_{b,p}(Y)$, the two definitions coincide, but for less regular functions they may differ.

Lastly, to precisely state the dual formulation, we define the set of admissible pairs $\Phi_C(\mu,\nu)$ as pairs of functions $(\phi,\psi)\in (-\mathcal L^1(\mu))\times \mathcal L^1 (\nu)$ satisfying 
$$\phi(x)-\rho(\psi)\leq C(x,\rho), $$
for all $x\in X$ and $\rho\in\P_p(Y)$ such that $\psi\in \mathcal L^1(\rho)$.

We can now state the main attainment result. 

\begin{theorem} \label{thmatt}
    Let $X$ be a Polish space, $\mu\in\PP(X)$ and $\nu\in\PP_p(\R^d)$. Assume that $C:X\times\PP_p(\R^d)\to[0,\infty]$ is lsc and both convex and  $\lc$-decreasing in the second argument.  If $C$ satisfies conditions (\ref{condB}) and (\ref{condC}), then
   \begin{equation}
        \inf_{\pi\in\Pi(\mu,\nu)}\int C(x,\pi_x)d\mu(x)=\sup_{\substack{(\psi^C,\psi)\in\Phi_C(\mu,\nu) \\ \psi \textnormal{ convex}}} \mu(\psi^C)-\nu(\psi) \label{att_eq}
   \end{equation}
and the supremum is attained. 
\end{theorem}
\begin{proof}
    The equality in (\ref{att_eq}) is an immediate consequence of \cite[Theorem 2.2]{BePaRiSc25} and Theorem \ref{mainthm}.

We now proceed to show that the supremum is attained. Let us call $(\phi,\psi)\in\Phi_C(\mu,\nu)$ an \textit{optimal dual pair} if $$\mu(\phi)-\nu(\psi)=\sup_{\substack{(\phi,\psi)\in\Phi_C(\mu,\nu)}} \mu(\phi)-\nu(\psi).$$
By \cite[Theorem 2.2]{BePaRiSc25}, an optimal dual pair $(\phi,\psi)\in\Phi_C(\mu,\nu)$ exists.
Assume that  $((\conv\psi)^C,\conv \psi)$ is also an optimal dual pair. Then:
$$\mu((\conv\psi)^C)-\nu(\conv \psi)=\sup_{\substack{(\phi,\psi)\in\Phi_C(\mu,\nu)}} \mu(\phi)-\nu(\psi)\geq\sup_{\substack{(\psi^C,\psi)\in\Phi_C(\mu,\nu)\\\psi \textnormal{ convex}}} \mu(\psi^C)-\nu(\psi),$$
which implies that the supremum in (\ref{att_eq}) is attained.

We are left to verify that $((\conv\psi)^C,\conv \psi)$ is an optimal dual pair. We do this in three steps. First, we show that $(\phi,\convr\psi)$ is an optimal dual pair, where $\convr$ is an appropriately defined approximation of the convex hull operator. Then we pass to the limit to show that $(\phi,\conv\psi)$ is also an optimal dual pair. Finally, we show that $\phi$ may be replaced by $(\conv\psi)^C$ with no ill effect on the optimality. 

\textit{Step 1:} $(\phi,\convr\psi) $ \textit{is an optimal dual pair}

\noindent Let us first define the operator $\convr$. For $R\in \mathbb N$, let $$\convr\psi(y):=\inf\left\{ \xi(\psi)\midauto \xi\in\P(B_R(y)),\, \delta_y\lc\xi,\, |\textnormal {supp}(\xi)|\leq d+1\right\}.$$ 
Compare this with the following representation of the convex hull:
\begin{equation}
\label{conv_hull_gen}
\conv\psi(y)=\inf\left\{ \xi(\psi)\midauto \xi\in\P(\R^d),\, \delta_y\lc\xi,\, |\textnormal {supp}(\xi)|\leq d+1\right\}.
\end{equation} 

Note that $\convr\psi$ is \textit{not} necessarily a convex function, but one has $\conv\psi \leq\convr\psi\leq\psi$ and $\convr\psi\downarrow\conv\psi$ as $R\to\infty$. Since $\convr\psi\leq\psi$, optimality of $(\phi,\convr\psi)$ will follow as soon as we show that $(\phi,\convr\psi) \in\Phi_C(\mu,\nu).$

Let us show that $\convr\psi \in \mathcal L^1(\nu)$. Assume that $\convr\psi(y)=-\infty$ for some $y\in \R^d$. This implies that $\conv\psi(y)=-\infty$ and thus by (\ref{conv_hull_gen}) we can find a sequence of probability measures $\{\xi_k\}_{k=1}^\infty$, each supported on at most $d+1$ points with mean $y$, such that $\lim_k \xi_k(\psi)=-\infty$. Fix $x_0\in X$ such that $\phi(x_0)\in\R$. Then we use the fact that $(\phi,\psi)\in\Phi_C(\mu,\nu)$ and the $\lc$-monotonicity of $C$ to get
$$\phi(x_0)-\xi_k(\psi)\leq C(x_0,\xi_k)\leq C(x_0,\delta_y)\leq a(x_0)+b(y).$$
Then the left-hand side tends to $\infty$, but the right-hand side is finite, leading to a contradiction. Thus we have shown that $\convr\psi$ is $(-\infty,\infty]$-valued. Note that our proof also shows that $\conv\psi$ is $(-\infty,\infty]$-valued, and hence a proper convex function. By \cite[Corollary 12.1.2]{rockafellar1997convex}, $\conv \psi$ admits an affine minorant  $\textnormal{aff}:\R^d\to\mathbb R$, so
$$\textnormal{aff}(y)\leq\conv\psi(y)\leq \convr\psi(y)\leq\psi(y),$$
but since $\psi$ and $\textnormal{aff}$ are both $\nu$-integrable, so is $\convr\psi$ (and $\conv \psi$).

Let us now show that the defining inequality of $\Phi_C(\mu,\nu)$ is satisfied by $(\phi,\convr\psi)$. Fix $\varepsilon>0$, $x\in X$ such that $\phi(x)\in\R$ and $\rho \in \P_p(\R^d)$ with $\convr\psi \in \mathcal L^1(\rho)$. By the von Neumann uniformization theorem, we can extract a measurable collection of probability measures $\{\xi_y\}_{y\in\R^d}\subseteq\P_p(\R^d)$ such that \begin{equation}
    \label{meas_sel_ineq}\convr\psi(y)\geq\xi_y(\psi)-\varepsilon, \;\, \delta_y\lc\xi_y, \;\,\xi_y(B_R(y))=1.
\end{equation}
Define $d\widetilde{\rho}(z):=d\xi_y(z) d\rho(y).$ Then $\widetilde{\rho}\ \in \P_p(\R^d)$ since $\widetilde{\rho}(|\cdot|^p)\leq2^{p-1}[{\rho}(|\cdot|^p)+R^p]$. Integrate the inequality in (\ref{meas_sel_ineq}) with respect to $\rho$ to get 
\begin{equation}
\label{ineq_meas_se}
\rho(\convr\psi)\geq\widetilde\rho(\psi)- \varepsilon.
\end{equation}
Note that $\psi\in\mathcal L^1(\widetilde\rho)$ since $\widetilde{\rho}(|\psi|)\leq \rho(\convr\psi)+\varepsilon+2\widetilde{\rho}(\textnormal{aff}^-)<\infty.$ 

Thus, we have
$$\phi(x)-\rho(\convr\psi)\leq \phi(x)-\widetilde\rho(\psi)+\varepsilon\leq C(x,\widetilde{\rho})+\varepsilon\leq C(x,\rho)+\varepsilon,$$
where the inequalities follow from (\ref{ineq_meas_se}), the fact that $(\phi,\psi)\in \Phi_C(\mu,\nu)$ and $\lc$-monotonicity of $C$, respectively. 

This shows that $(\phi,\convr\psi)\in\Phi_C(\mu,\nu)$, which immediately implies that $(\phi,\convr\psi)$ is an optimal dual pair.

\textit{Step 2:} $(\phi,\conv\psi) $ \textit{is an optimal dual pair}

\noindent
As before, since $\conv\psi\leq\psi$, optimality will follow as soon as $(\phi,\conv\psi)\in\Phi_C(\mu,\nu)$. We already verified that $\conv\psi\in\mathcal L^1(\nu)$ in the previous step, so it remains to show that the defining inequality of $\Phi_C(\mu,\nu)$ holds for the pair $(\phi,\conv\psi)$.

Fix $x\in X$ such that $\phi(x)\in\R$ and $\rho \in \P_p(\R^d)$ with $\conv\psi \in \mathcal L^1(\rho)$. By Egorov's theorem, we can find a sequence $\{A_k\}^{\infty}_{k=1}$ of increasing compact subsets of $\R^d$ such that $\rho(\bigcup_kA_k)=1$ and $\convr\psi\to\conv\psi$ uniformly on each set $A_k$. Define for each $k\in\mathbb N$ a probability measure $\rho_k:=(\rho|{A_k}) / \rho(A_k).$

Then by dominated convergence, we have
$$\phi(x)-\rho(\conv\psi)=\lim_k\phi(x)-\rho_k(\conv\psi),$$
but by uniform convergence we also get
$$\lim_k\phi(x)-\rho_k(\conv\psi)=\lim_k \lim_R\phi(x)-\rho_k(\convr\psi)\leq \limsup_k C(x,\rho_k)\leq C(x,\rho),$$
where the inequalities above are due to $(\phi,\convr\psi)\in\Phi_C(\mu,\nu)$ and the continuity condition (\ref{condC}) with $Y_k:=A_k\cup(\R^d\setminus\bigcup_{k=1}^\infty A_k)$.

\textit{Step 3:} $((\conv\psi)^C,\conv\psi) $ \textit{is an optimal dual pair}

\noindent
Since $(\phi,\conv\psi)\in\Phi_C(\mu,\nu)$, we have for all $x\in X$ and $\rho \in \P_p(\R^d)$ such that $\conv\psi \in\mathcal L^1(\rho)$ 
$$\phi(x)\leq \rho(\conv\psi)+C(x,\rho),$$
and taking the infimum over all such $\rho$ yields $\phi(x)\leq (\conv \psi)^C(x)$. On the other hand, by the definition of the convex conjugate, we have 
$$(\conv\psi)^C(x)\leq \rho(\conv\psi)+C(x,\rho)$$
for all $x\in X$ and $\rho\in\P_p(\R^d)$ such that $\conv \psi \in \mathcal L^1(\rho)$. 

By $\nu$-integrability of $\conv\psi$  and assumption (\ref{condB}), we get
$$\phi(x)\leq(\conv\psi)^C(x)\leq \nu(\conv\psi)+a(x)+\nu(b),$$
 which proves that $(\conv\psi)^C\in(-\mathcal L^1(\mu))$. 
 
 Hence $((\conv\psi)^C,\conv\psi)\in \Phi_C(\mu,\nu)$ and since $\phi\leq (\conv\psi)^C$, we have that $((\conv\psi)^C,\conv \psi)$ is an optimal dual pair.

\end{proof}
The proof given above, \textit{mutatis mutandis}, also shows the increasing convex version of Theorem \ref{thmatt}.

\begin{theorem} \label{thmatt2}
    Let $X$ be a Polish space, $\mu\in\PP(X)$ and $\nu\in\PP_p(\R^d)$. Assume that $C:X\times\PP_p(\R^d)\to[0,\infty]$ is lsc and both convex and  $\lic$-decreasing in the second argument. If $C$ satisfies conditions (\ref{condB}) and (\ref{condC}), then
   \begin{equation}
        \inf_{\pi\in\Pi(\mu,\nu)}\int C(x,\pi_x)d\mu(x)=\sup_{\substack{(\psi^C,\psi)\in\Phi_C(\mu,\nu) \\ \psi \textnormal{ increasing convex}}} \mu(\psi^C)-\nu(\psi)\label{att_eq2}
   \end{equation}
   and the supremum is attained. 
\end{theorem}


\begin{remark} \label{lowerboundremark}
    The assumption that $C$ is nonnegative in Theorem \ref{thmatt} can be weakened. Indeed, as in Theorem \ref{thmmod1}, it suffices to have $C(x,\rho)\geq-(a_{\ell}(x)+\rho(b_{\ell}))$ for $a_{\ell}\in L^1(\mu)\cap \LSC(X)$ and some convex function $b_{\ell}\in\mathcal{C}_{b,p}(\R^d)$. 
\end{remark}

\section{Applications}\label{APP}

\subsection{Convex Dual Potentials}

\subsubsection{Barycentric Costs}

\label{barcos}
A cost function that depends on the second argument solely via its mean/\textit{barycenter}, is called barycentric.
In \cite{GoJu18}, \text{Gozlan} and \text{Juillet} consider the dual formulation associated to such costs and restrict it to convex functions. We recover this restriction result using Theorem \ref{mainthm}.

Let $X=\R^d$ and $p=1$, and take $\mu,\nu\in\P_1(\R^d).$ We consider the cost function $C(x,\rho)=\theta(x-\mean(\rho))$ for some convex $\theta: \R ^d \to [0,\infty)$, where $\mean(\rho):=\int_{\R^d} y d\rho(y)$. Note that $\rho_1\lc\rho_2$ implies $\mean(\rho_1)=\mean(\rho_2)$, which shows that $C$ is $\lc$-decreasing, so by Theorem \ref{mainthm}, we get
$$\inf_{\pi\in\Pi(\mu,\nu)} \int \theta(x-\mean(\pi_x))d\mu(x)= \sup_{\substack{\psi \in {\mathcal{C}}^{}_{b,1}(\R^d) \\ \psi\text{ convex}}} \mu(\psi^C) - \nu(\psi).$$
This recovers {\cite[Theorem 1.1]{GoJu18}}. It is worth noting that the $C$-conjugate in the expression above simplifies to $\psi^C(x)=
\inf_{z\in\R^d}\psi(z)+\theta(x-z).$

For concreteness, taking $\theta(x)=\left \| x \right \|$, for $\left \| \cdot \right \|$ a norm on $\R^d$, produces the convex Kantorovich--Rubinstein duality, while $\theta(x)=|x|^2$ with $\mu,\nu\in\PP_2(\R^d)$ recovers the Brenier--Strassen mixture result \cite[Theorem 1.2]{GoJu18}. Moreover, Theorem \ref{thmatt} ensures that in both of these instances the dual is attained by a $\nu$-integrable convex function $\psi^{\text{opt}}:\R^d\to(-\infty,\infty]$.

\subsubsection{Martingale Benamou--Brenier} \label{sbm}
The martingale analogue of the Benamou--Brenier problem \cite{backhoffveraguas2025existencebassmartingalesmartingale}  turns out to be equivalent to a weak transport problem where the dual formulation can be restricted to convex functions. We reproduce this result using Theorem \ref{thmmod1}.

Let $X=\R^d$, $p=2$, $a_\ell(x)=0$ and $b_\ell(y)=1/2|y|^2\in{\mathcal{C}}_{b,2}(\R^d)$. We consider the following cost function:
$$C(x,\rho)=  \begin{cases} 
        -\text{MCov}(\rho,\gamma), & x=\mean (\rho) \\
      \infty, & \text{otherwise}, 
   \end{cases}
  $$
where 
$\text{MCov}(\rho_1,\rho_2):=\max_{q \in \Pi(\rho_1,\rho_2) } \int_{\R^d\times \R ^d}yz dq(y,z)$ and $\gamma$ denotes the standard normal distribution on $\R^d$.
 
Let $\mu,\nu\in\PP_2(\R^d)$ with $\mu\lc\nu$. One readily checks the assumptions of Theorem \ref{thmmod1}, hence we have
$$\sup_{\pi\in\text{MT}(\mu,\nu)} \int \text{MCov}(\pi_x,\gamma)d\mu(x)= \inf_{\substack{{\psi \in \mathcal Q} \\\psi \text{ convex} }} \nu(\psi) -\mu(\psi^C),$$
where $\text{MT}(\mu,\nu)$ denotes the set of martingale couplings between $\mu$ and $\nu$, and $\mathcal{Q}:=( {\mathcal{C}}^{}_{b,2}(\R^d)+1/2|\cdot|^2)$. Thus we have recovered \cite[Proposition 3.5]{backhoffveraguas2025existencebassmartingalesmartingale}.

\begin{remark} \label{nonintdual}
    We note that dual attainment is a subtle issue in the case of martingale transport, which in general does not satisfy the boundedness assumption (\ref{condB}). See \cite{backhoffveraguas2025existencebassmartingalesmartingale} for a detailed study of dual attainment for the martingale Benamou--Brenier problem and, in particular, \cite[Example 6.7]{backhoffveraguas2025existencebassmartingalesmartingale} for an example where there is no optimal $\nu$-integrable dual potential.
    \end{remark}
    \begin{remark} In \cite[Theorem 5.4]{BePaRiSc25}, the authors relax the martingale constraint by considering the weak transport problem with the cost function $C(x,\rho):=\beta|x-\mean (\rho)|^2-\alpha\text{MCov}(\rho,\gamma)$ for $\alpha,\beta>0.$ In this relaxed setting, dual attainment follows from Theorem \ref{thmatt} and Remark \ref{lowerboundremark}. Note that letting $\beta\to\infty$ formally recovers the martingale constraint.
\end{remark}

\subsubsection{Classical Optimal Transport} We highlight the connection with classical transport and apply our results to the celebrated quadratic cost function. 

Let $X=\R^d$,  $p=2$, $a_\ell(x)=\frac{x^2}{2} , b_\ell(y)=\frac{y^2}{2}$ and take $\mu,\nu\in\P_2(\R^d)$. We consider the cost function $C(x,\rho)=\int-xyd\rho(y)$. By Theorem \ref{thmmod1} we have
$$\sup_{\pi\in\Pi(\mu,\nu)}\int xy d\pi(x,y)=\inf_{\substack{\psi\in\mathcal Q\\\psi \text{ convex}}}\mu(\psi^*)+\nu(\psi),$$
where $\psi^*(x):=\sup_{y\in\R^d}xy-\psi(y)$ denotes the convex conjugate of $\psi$. This recovers the classical duality in \cite[Theorem 5.10]{villani2008optimal} for $c(x,y)=-xy.$ The dual is attained by a $\nu$-integrable convex function $\psi^{\text{opt}}:\R^d\to(-\infty,\infty]$. 

\begin{remark}
    More generally, in the context of classical transport---when $C(x,\rho)=\int c(x,y)d\rho(y)$---the condition that $C$ is $\lc$-decreasing in the second argument is equivalent to $c$ being concave in the second argument.
\end{remark}
\subsubsection{Strassen's Theorem for Martingales}
Recall the classical result \cite{Strassen1965} of Strassen on the existence of martingale couplings.
\begin{theorem}[Strassen] \label{strass} Let $\mu,\nu\in\PP_1(\R^d)$. Then $\mu\lc\nu$ if and only if there exists a martingale coupling between $\mu$ and $\nu$, that is, for some $\pi\in\Pi(\mu,\nu)$ we have $x=\textnormal{mean}(\pi_x)$ for $\mu$-a.e. $x\in \R^d$.
    
\end{theorem}
\begin{proof}
    By applying Theorem \ref{mainthm} to $$C(x,\rho):=\begin{cases}
        0 \quad\;\; \delta_x \lc\rho,\\
        \infty \quad \text{otherwise,}
    \end{cases}$$
we get

$$\inf_{\pi\in\Pi(\mu,\nu)} \int \chi_{\delta_x\lc\pi_x} d\mu(x)= \sup_{\substack{\psi \in {\mathcal{C}}^{}_{b,1}(\R^d) \\ \psi\text{ convex}}} \mu(\psi) - \nu(\psi).$$
    The left-hand side is zero exactly when a martingale coupling exists, whereas the right-hand side is zero exactly when $\mu\lc\nu$. This proves the claim.
\end{proof}

Although this strategy for proving Strassen's theorem is identical to that of \cite[Theorem 3.1]{GoRoSaTe14}, the novel contribution lies in the fact that it suffices to check the assumptions of Theorem \ref{mainthm} to ensure that the dual can be taken over convex functions. 

\subsection{Increasing Convex Dual Potentials}

\subsubsection{Multiple-Good Monopolist} \label{monopol}
Originally discussed in \cite{DaDeTz17} by \text{Daskalakis, Deckelbaum,} and\text{ Tzamos}, and later in
\cite{backhoffveraguas2020applicationsweaktransporttheory} by \text{Backhoff} and \text{Pammer}, the multiple-good monopolist problem was shown to admit a weak optimal transport formulation. The result \cite[Theorem 6.1]{backhoffveraguas2020applicationsweaktransporttheory} shows that the dual problem can be restricted to convex and coordinatewise increasing functions. We recover this result as a special case of Theorem \ref{mainthmic}.

Let $X=\R^d$ and fix $\theta\in{\mathcal{C}}_{b,{\color{black}p}}(\R^d)$ convex and take $\mu,\nu\in\P_p(\R^d)$.  We consider  $$C(x,\rho):=\inf_{\xi\lic\rho}\theta(x-\mean(\xi)).$$ 
The function $C$ is lsc, convex in the second argument, bounded from below and, crucially,  $\lic$-decreasing. By Theorem \ref{mainthmic}  we have

$$\inf_{\pi\in\Pi(\mu,\nu)} \int \inf_{\xi\lic\pi_x}\theta(x-\mean(\xi))d\mu(x)= \sup_{\substack{\psi \in {\mathcal{C}}^{}_{b,p}(\R^d) \\ \psi\text{ inc. convex}}} \mu(\psi^C) - \nu(\psi),$$
which recovers \cite[Theorem 6.1]{backhoffveraguas2020applicationsweaktransporttheory}. 

Moreover, by Theorem \ref{thmatt2}, we have that the dual is attained by a $\nu$-integrable increasing convex function $\psi^{\text{opt}}:\R^d\to(-\infty,\infty]$. 

\subsubsection{Increasing Convex Kantorovich--Rubinstein}
Let $X=\R^d$,  $\|\cdot\|_q$ be the $\ell^q$ norm on $\R^d$ for $q\in[1,\infty]$. Consider the cost function $C(x,\rho):=\|(x-\mean(\rho))_+\|_q$, where $(v)_+$ denotes the coordinatewise positive part of $v\in \R^d$. By Theorem \ref{mainthmic}, we have for all $\mu,\nu\in\PP_1(\R^d)$:
$$\inf_{\pi\in\Pi(\mu,\nu)} \int \|(x-\mean(\pi_x))_+\|_q d\mu(x)=\sup_{\substack{\psi \; \|\cdot\|_q\text{-Lipschitz}\\\psi\text{ inc. convex}}} \mu(\psi^C)-\nu(\psi).$$
One then replaces $\psi$ by $\psi^C$, which is $1$-Lipschitz and remains increasing and convex. So up to renaming we get
$$\inf_{\pi\in\Pi(\mu,\nu)} \int \|(x-\mean(\pi_x))_+\|_q d\mu(x)=\sup_{\substack{\psi \;1\text{-}\| \cdot\|_q\text{-Lipschitz}\\\psi\text{ inc. convex}}} \mu(\psi)-\nu(\psi).$$
By Theorem \ref{thmatt2}, the dual problem is attained by an increasing convex function $\psi^{\textnormal {opt}}:\R^d\to(-\infty, \infty ]$.

\subsubsection{Strassen's Theorem for Submartingales}
In the same spirit as Theorem \ref{strass}, we recover Strassen's theorem for submartingales as a direct consequence of our main result for increasing convex functions.
\begin{theorem}[Strassen] \label{strass2} Let $\mu,\nu\in\PP_1(\R^d)$. Then $\mu\lic\nu$ if and only if there exists a submartingale coupling between $\mu$ and $\nu$, that is, for some $\pi\in\Pi(\mu,\nu)$ we have $x\leq\textnormal{mean}(\pi_x)$ for $\mu$-a.e. $x\in \R^d$.
    
\end{theorem}

 \section{General Result}\label{nesto}

The final section is devoted to developing a general framework for the restriction of the dual problem. Let $(X,d_X)$ and $(Y,d_Y)$ be Polish metric spaces and $p\geq1$. We write $\LSC_p(Y)$ for the space of lsc functions on $Y$ that are bounded in absolute value by a multiple of $1+d_Y(\cdot,y_0)^p$ for some $y_0\in Y$.

   We call $E\subseteq \LSC_{p}(Y)$ a \emph{stable cone of functions}  if
\begin{enumerate}
    \item \label{ass1} $E$ contains all constant functions,
    \item \label{ass2} $E$ is a convex cone,
    \item \label{ass3} if $\{f_i\}_I\subseteq E$ is such that $\sup_I f_i \in \LSC_p(Y)$, then $\sup_I f_i\in E$.
    \item \label{ass4} for any $f\in E$, there is a collection of \textit{continuous} functions $\{f_i\}_I\subseteq E$ such that $f=\sup_I f_i$,
    \item  \label{ass5} we have $\sup\{\phi(y)\mid \phi\leq f, \phi\in E\}=\inf\{\rho(f)\mid\delta_y\ld \rho \}$ for $f\in{\mathcal{C}}_{b,p}(Y)$.
\end{enumerate}
We say that $\rho_1, \rho_2 \in\PP_p(Y)$ are in $E$-order, in signs $\rho_1\preceq_E \rho_2$, if $\rho_1(f) \leq \rho_2(f)$ for all $f\in E$. 

The proof of Theorem \ref{mainthm} naturally extends to the setting of stable cones.
\begin{theorem}\label{mainthmgen}
     Let $E$ be a stable cone of functions on $Y$, $\mu\in\PP(X)$, $\nu \in \PP_p(Y)$, and $C:X\times\PP_p(Y)\to[0,\infty]$ be lsc and convex in the second argument. If $C$ is $\preceq_E$-decreasing in the second argument, then 
\begin{align} \label{mainthmdisplgen}
		\inf_{\pi\in\Pi(\mu,\nu)} \int C(x,\pi_x)\, d\mu(x)= \sup_{\substack{\psi \in E\\ \psi \textnormal{ lower bdd.}}} \mu(\psi^C) - \nu(\psi).
	\end{align}
    Conversely, if (\ref{mainthmdisplgen}) holds for all $\mu\in\PP(X)$, $\nu\in\PP_p(Y)$, then $C$ is $\preceq_E$-decreasing in the second argument.
\end{theorem}

\begin{remark}
    Note that by taking $E=\{f\in\LSC_p(\R^d): f\text{ convex}\}$ and $E=\{f\in\LSC_p(\R^d): f\text{ increasing convex}\}$ we recover Theorem \ref{mainthm} and Theorem \ref{mainthmic}.
\end{remark} 

We say that $E\subseteq \LSC_p(Y)$ is a \textit{semi-stable cone} if it satisfies assumptions (\ref{ass1})--(\ref{ass4}); these are natural requirements reflecting the basic properties of (increasing) convex functions. The somewhat artificial assumption (\ref{ass5}) is a statement about the dual representation of the $E$-hull, a generalization of the convex hull. It is classical that both the convex hull and the increasing convex hull satisfy such a representation.  Furthermore, if $Y$ is a \textit{compact} Polish space, and $E$ is a semi-stable cone, then by \cite[Corollary XI, T46] {meyer1966probability} it is a stable cone. Despite \cite{Gogus2013} asserting that the same conclusion remains true when  $Y$ is \textit{locally compact}, \cite{Khabibullin2021} points out a flaw in the argument of \cite{Gogus2013}, and \cite{Djire_Wiegerinck_2019} provides a counterexample. 






In the compact setting, we  recover the generalization of Strassen's theorem to convex cones; these particular couplings are sometimes referred to as \textit{dilations} (see \cite[Theorem 19.40]{aliprantis2007infinite}).

\begin{theorem}[Dilations]\label{dil} Assume that $Y$ is compact and  $E$ is a semi-stable cone of functions on $Y$. Let $\mu,\nu\in\PP(Y)$. Then $\mu\ld\nu$ if and only if there exists a coupling $\pi\in\Pi(\mu,\nu)$ satisfying $\delta_y \ld\pi_y$ for $\mu$-a.e. $y\in Y$.
\end{theorem}

\begin{proof}
    Apply Theorem \ref{mainthmgen} to $C:Y\times\P(Y)\to[0,\infty]$ given by $$C(x,\rho):=\begin{cases}
        0 \quad\;\; \delta_x \ld\rho,\\
        \infty \quad \text{otherwise,}
    \end{cases}$$
to get
$$\inf_{\pi\in\Pi(\mu,\nu)} \int \chi_{\delta_x\ld\pi_x}d\mu(x)= \sup_{\substack{\psi \in E \\ }} \mu(\psi) - \nu(\psi).$$
\end{proof}

Let us now recall the following relation for $\mu,\nu\in\P_2(\R^d)$:
$$\inf_{\eta\lc\nu} W^2_2(\mu,\eta)= \inf_{\pi\in\Pi(\mu,\nu)} \int_{\R^d} |x-\mean(\pi_x)|^2 d\mu(x),$$
where $W_2(\cdot,\cdot)$ is the 2-Wasserstein distance.
We prove an analogous equality for general cost functions on compact Polish spaces and relate it to our dual representation result.

\begin{theorem} \label{proj} Assume that $Y$ is compact and $E$ is a semi-stable cone of functions on $Y$. Let $C:X\times\PP(Y)\to[0,\infty]$ be lsc and convex in the second argument. Then for $\mu\in\PP(X)$, $\nu\in \PP(Y)$, we have
    $$\inf_{\eta\ld\nu}  \inf_{\pi\in\Pi(\mu,\eta) }\int C(x,\pi_x)d\mu(x)=\inf_{\pi\in\Pi(\mu,\nu)} \int \inf_{\xi\ld\pi_x}C(x,\xi) d\mu(x)=\sup_{\psi\in E}\mu(\psi^C)-\nu(\psi).$$
\end{theorem}
\begin{proof}
  Let us prove the first equality.  Fix $\varepsilon>0,$ and let $\eta^*\ld\nu$ and $\pi^*\in\Pi(\mu,\eta^*)$ be almost optimizers:
    \begin{align*}
        \inf_{\eta\ld\nu} \left[ \inf_{\pi\in\Pi(\mu,\eta) }\int C(x,\pi_x)d\mu(x)\right]+\varepsilon\geq
          \int C(x,\pi^*_x)d\mu(x).
    \end{align*}
By Theorem \ref{dil}, we have that $\eta^*\ld\nu$ implies the existence of a coupling $\widetilde{\pi}\in \Pi(\eta^*,\nu)$ that satisfies $\delta_{\widetilde{y}}\ld\widetilde{\pi}_{\widetilde{y}}$ for $\eta^*$-a.e. $\widetilde{y}\in Y$.

Define the probability measure $\pi$ on $X\times Y$ by $d\pi(x,y):=d\widetilde{\pi}_{\widetilde{y}}(y) d\pi^*(x,\widetilde{y}).$ One directly verifies that $\pi\in\Pi(\mu,\nu)$ and satisfies $\pi_x^*\ld\pi_x$ for $\mu$-a.e. $x\in X$. Hence
 \begin{align*}
          \int C(x,\pi^*_x)d\mu(x)\geq \int \inf_{\xi\ld\pi_x}C(x,\xi) d\mu(x) \geq \inf_{\pi\in\Pi(\mu,\nu)}\int\inf_{\xi\ld\pi_x}C(x,\xi) d\mu(x) .
    \end{align*}

For the converse inequality, fix $\varepsilon>0$ and let $\pi^*\in\Pi(\mu,\nu)$ be almost optimal and $\{\xi_x^*\}_{x\in X}$ a measurable selection of almost optimizers satisfying $\xi_x^*\ld \pi_x^*$. Then
$$\inf_{\pi\in\Pi(\mu,\nu)} \int \left[\inf_{\xi\ld\pi_x}C(x,\xi)\right] d\mu(x)+\varepsilon \geq \int C(x,\xi_x^*)d\mu(x).$$
Now define $d\eta(y):=\int_X d\xi_x^*(y)d\mu(x)$ and $d\pi(x,y):=d\xi^*_x(y)d\mu(x).$ Then $\eta\ld\nu$ and $\pi\in\Pi(\mu,\eta).$ This proves the first equality. 

To show the second equality, note that $\widehat{C}(x,\rho):=\inf_{\xi\ld\rho} C(x,\xi)$ is lsc (see \cite[Proposition 7.33]{bertsekas1978stochastic}), convex in $\rho$ and $\ld$-decreasing, so by Theorem \ref{mainthmgen}, we get
$$\inf_{\pi\in\Pi(\mu,\nu)} \int \inf_{\xi\ld\pi_x}C(x,\xi) d\mu(x)=\sup_{\psi\in E}\mu(\psi^{\widehat{C}})-\nu(\psi).$$
One easily shows that $\psi^{\widehat{C}}=\psi^C$ for $\psi\in E$, which finalizes the proof.
\end{proof}

    To illustrate the underlying idea of Theorem \ref{proj}, consider the following two-step process of transporting $\mu$ to $\nu$: first transport $\mu$ optimally (w.r.t. the cost function $C$)  to an intermediate measure $\eta$ that satisfies $\eta\ld\nu$, and then transport $\eta$ to $\nu$ via an $E$-dilation coupling, whose existence is guaranteed by Theorem \ref{dil}, with the intermediate measure $\eta$ chosen to minimize the transport cost between $\mu$ and $\eta$. The cost incurred by the first step (that is, the double infimum in Theorem \ref{proj}) is identical to the weak transport cost between $\mu$ and $\nu$ with respect to the $\ld$-decreasing hull of $C$, i.e. $\widehat{C}(x,\rho)=\inf_{\xi\ld\rho} C(x,\xi)$. Additionally, both values coincide with the dual formulation restricted to potentials in the semi-stable cone $E$.

    In the event that $C$ itself is $\ld$-decreasing, the cost function $\widehat{C}$ simplifies to $C$, the optimal $\eta$ is the target measure $\nu$, and the statement of Theorem \ref{proj} collapses to Theorem \ref{mainthmgen} for compact Polish spaces.

\bibliographystyle{plain}

\bibliography{joint_biblio, literatur}


\end{document}